\documentclass[11pt]{article}

\usepackage{makeidx}     
\usepackage{graphicx,epsfig,lscape,float}    
\usepackage{subfigure}

\usepackage{amsmath,array,amssymb, amsfonts, latexsym,amsthm, mathrsfs, dsfont}
\usepackage{psfrag,fancybox,authblk} 

\usepackage{multicol}   
\usepackage{enumerate}   
\usepackage{accents}  
\makeindex             
\newcommand{\R}{\mathbb{R}} 

\newcommand{\cx}{\check x}

\newcommand{\hx}{\hat x}

\newcommand{\rx}{\accentset{\circ}x}

\renewcommand{\epsilon}{\varepsilon}

\usepackage{amsmath,array,amssymb,amsfonts,graphicx}
\usepackage{subfigure}

\newcommand{\abs}{{\mbox{abs}}}

\newcommand{\norm}[1]{\left\lVert #1 \right\rVert}
\newcommand{\betrag}[1]{\left\lvert #1 \right\rvert}

\newcommand{\incf}[1]{\Delta F(#1)}
\newcommand{\integ}[1]{\int^\frac12_{-\frac12} #1 \mathrm dt}

\newtheorem{proposition}{Proposition}[section]
\newtheorem{theorem}[proposition]{Theorem}

\newtheorem{lemma}[proposition]{Lemma}

\newtheorem{defi}{Definition}[section]
\newtheoremstyle{custom}{5pt}{5pt}{}{}{\bfseries}{:}{5pt}{}
\theoremstyle{custom}

\begin{document}


\title{Integrating Lipschitzian Dynamical Systems
using Piecewise Algorithmic Differentiation\footnote{Our notion of linearity includes nonhomogeneous functions, where  
the adjective {\em affine} or perhaps {\em polyhedral} would be more precise. However, such mathematical terminology might be less appealing to 
computational practitioners and to the best of our knowledge  there are no good nouns corresponding to {\em linearity} and {\em linearization} for the adjectives affine and polyhedral.}}

\author[1]{Andreas Griewank}
\author[2]{Richard Hasenfelder}
\author[2]{Manuel Radons}
\author[3,2]{Tom Streubel}
\affil[1]{School of Information Sciences Yachaytech, Ecuador}
\affil[2]{Humboldt University of Berlin, Germany}
\affil[3]{Zuse Institute Berlin, Germany}

\date{\scriptsize{\textit{griewank@yachaytech.edu.ec,
hasenfel@math.hu-berlin.de, radons@math.hu-berlin.de, streubel@zib.de}}}



\maketitle

\noindent 
\begin{abstract}
In this article we analyze a generalized trapezoidal rule for initial value problems with piecewise smooth right hand side \(F:\R^n\to\R^n\). When applied to such a problem the classical trapezoidal rule suffers from a loss of accuracy if the solution trajectory intersects a nondifferentiability of \(F\). The advantage of the proposed generalized trapezoidal rule is threefold: Firstly we can achieve a higher convergence order than with the classical method. Moreover, the method is energy preserving for piecewise linear Hamiltonian systems. Finally, in analogy to the classical case we derive a third order interpolation polynomial for the numerical trajectory. In the smooth case the generalized rule reduces to the classical one. Hence, it is a proper extension of the classical theory. An error estimator is given and numerical results are presented. 
\end{abstract}

\noindent 
{\bf Keywords}
Automatic Differentiation, Lipschitz Continuity, Piecewise Linearization, Nonsmooth, Trapezoidal Rule, Energy Preservation, Continuous Output
 
\section{Introduction}
Many realistic computer models are nondifferentiable in that the functional relation between input and output variables is not smooth. We are particularly focusing on Lipschitz continuous models where the nondifferentiabilities have a special structure. We present a technique that handles such functions in the context of the numerical solution of ordinary differential equations (ODE's). It is based on algorithmic differentiation (AD) and generalizes this concept. For further information on the general theory of AD we refer to \cite{buch, Naumann2012}. \\
Consider the following initial value problem for an autonomous ODE.
	\begin{equation}\label{Eqn:IVP}
		\dot x(t) = F(x(t)),\;\;\;\; x(0) = x_0\,,
	\end{equation}
where $F:\R^n\rightarrow \R^n$ is assumed to be locally Lipschitz continuous. It is well known that this system has a unique local solution up to some time $\bar t > 0$.	For a time-step $h>0$ the exact solution of \eqref{Eqn:IVP} satisfies
	\begin{equation*}\label{Eqn:ODEIntegrated}
	      \hx = \cx + \int_0^h F(x(t)) \mathrm dt,
	\end{equation*}
with $\hx = x(h)$ and $\cx = x_0$. In general the integral cannot be evaluated exactly. 

In the derivation of the classical trapezoidal rule a linear approximation of the right hand side is utilized. The integration of these approximations yields a third order local truncation error if \(F\) is smooth. If \(F\) is only Lipschitz continuous, the truncation error will drop to second order where the solution trajectory intersects a nondifferentiability. 

The key idea to reestablish a third order truncation error everywhere is to approximate \(F\) by a \textbf{piecewise linear} function that reflects the structure of the nondifferentiabilities of \(F\). Employing this approach we will construct a generalized trapezoidal rule with the following three major benefits: 
\begin{itemize}
\item We achieve second order convergence in general and third order via Richardson extrapolation along solution trajectories with finitely many kink locations. 
\item A third order interpolating polynomial as continuous approximation of the trajectory will be given.
\item The method is energy preserving on piecewise linear Hamiltonian systems. 
\end{itemize}
\subsection*{Content and Structure}
The article is organized as follows: 
In Section 2 we introduce the necessary prerequisites from piecewise linear theory and generalized algorithmic differentiation. In Section 3 the generalized trapezoidal rule is constructed and convergence results are proved. The extrapolation results are presented in Section 4, as well as the geometric integration properties. The error estimator is derived in Section 5. The sixth section contains numerical results. We conclude with some final remarks.

\section{Piecewise Linear Model}\label{PLModel}
	\begin{defi}
		A continuous function $F:\R^n\rightarrow\R^m$ is called {\bf piecewise linear} if there exists a finite number of affine {\bf selection functions} $F_i:\R^n\rightarrow\R^m$ such that at any given $x\in \R^n$ there exist at least one index $i$ with $F(x)=F_i(x)$.
	\end{defi}
%
	Let the index set $I =\lbrace 1,...,k \rbrace$ of the selection functions be given.
	According to \cite[Prop. 2.2.2]{scholtespl} we can find subsets $M_1,...,M_l\subset I$ such that a scalar valued piecewise linear function $f$ can be represented as $$f(x) = \max_{1\leq i\leq l}\min_{j\in M_j} f_j(x)\, .	$$	This concept, which is called {\bf max-min representation}, naturally carries over to vector valued functions $F$, where we can find such a decomposition for every component of the image.
	
	Note that piecewise linear functions are globally Lipschitz continuous.
	For a further discussion of their properties we refer to \cite{scholtespl}. 	
	Next we consider continuous, piecewise differentiable functions $F$ 
	that can be computed
	by a finite program called \emph{evaluation procedure}.
	   An evaluation procedure is a composition 
of so-called elementary functions which make up the atomic constituents of more complex functions.
Basically the selection of elementary functions for the library is arbitrary, as long as they comply with assumption (ED) (elementary differentiability, in \cite{buch}), meaning that they are at least once Lipschitz-continuously differentiable. 
Common examples are: 
\[ \tilde \Phi\ :=\ \{ +, -, \ast, /, \sin, \cos, \tan, \cot, \exp, \log, \dots \}\; . \]

	In our case, we will allow the evaluation procedure of $F: D\subseteq\R^n\rightarrow\R^m$ to contain,
	in addition to the usual smooth elementary functions, the absolute value $\abs(x) = |x|$, i.e., our library is of the form $$\Phi\ := \tilde \Phi\ \cup\ \{\abs \}\; .$$
	Consequently, we can also handle the maximum and minimum of two values via the representation
	\begin{equation*}
		\max(u,v)=(u+v+|u-v|)/2 ,~~ \min(u,v)=(u+v-|u-v|)/2\; .
	\end{equation*}
	We call the resulting functions {\em composite piecewise differentiable}. These functions are locally Lipschitz continuous and almost everywhere differentiable in the classical sense. Furthermore they are differentiable in the sense of Bouligand and Clarke, cf. \cite{clarke}. 
	The evaluation procedure of $y=F(x)$ can be interpreted as a directed, acyclic
	graph from $x = (v_{1-n},...,v_0)$ to $y = (v_{l-m+1},...,v_{l})$,
	where the intermediate values $v_i, i=1,...,l$ are computed by binary operations
	$v_i = v_j\circ v_k$ with $\circ\in\lbrace +,-,* \rbrace$ and $v_j,v_k \prec v_i$
	or unary functions $v_i = \phi_i(v_j)$ with $v_j\prec v_i$, where $\phi_i\in \Phi$.
 The relation $\prec$ represents the
	data dependence in the graph of the evaluation procedure, which must be acyclic. 
    
    We now want
	to compute an incremental approximation $\Delta y = \Delta F(\rx;\Delta x)$ to $F(\rx + \Delta x)-F(\rx)$
	at a given $\rx$ and for a variable increment $\Delta x$.
	Assuming that all functions other than the absolute value
	are differentiable, we introduce the propagation rules
	\begin{equation}\label{Eqn:TangentLinearization}
		\begin{array}{llll}
			\Delta v_i=\Delta v_j\pm \Delta v_k  & \mbox{for }  \mathring{v}_i=\mathring{v}_j\pm \mathring{v}_k\; , &\\
			\Delta v_i=\mathring{v}_j\ast\Delta v_k+\Delta v_j\ast \mathring{v}_k  & \mbox{for } \mathring{v}_i=\mathring{v}_j\ast \mathring{v}_k\; , &  \\
			\Delta v_i=\mathring{c}_{ij}\Delta v_j \;\;\;\;\; \mbox{with } \mathring{c}_{ij}=\varphi'(\mathring{v}_j) & \mbox{for } \mathring{v}_i=\varphi_i(\mathring{v}_j)\neq \abs(\cdot)\; , \\
			\Delta v_i=\abs(\mathring{v}_j+\Delta v_j)-\abs(\mathring{v}_j)  &  \mbox{for } \mathring{v}_i=\abs(\mathring{v}_j)\; .
		\end{array}
	\end{equation}
	Whenever $F$ is globally differentiable (i.e., there are no $\abs$ calls
	in the evaluating procedure) we get $\Delta y = F'(\rx)\Delta x$,
	where $F'(\rx)\in \R^{m\times n}$ is the Jacobi-matrix.
	
	Note, that the propagation rules \eqref{Eqn:TangentLinearization} rely on the
	so called tangent approximation of $F$ at a certain point $\rx$. However,
	there are applications of piecewise linearization (especially concerning ODE integration)
	where one wants to consider approximations of $F$ based on secants.
	Given two points $\check x, \hat x$ we compute $\mathring x = (\check x + \hat x)/2$
	and $\mathring F= (F(\check x) + F(\hat x))/2$. Now we consider the
	secant approximation of $F$.
	\begin{equation}\label{Eqn:SecantApproximation}
		F(x) \approx \mathring F + \Delta F(\hat x,\check x; x-\mathring x)
	\end{equation}
	The essential features of the two piecewise linearization modes are displayed in in Figure \ref{Fig:Modes}.
	\begin{figure}[htp]
	  \centering

\subfigure{\includegraphics{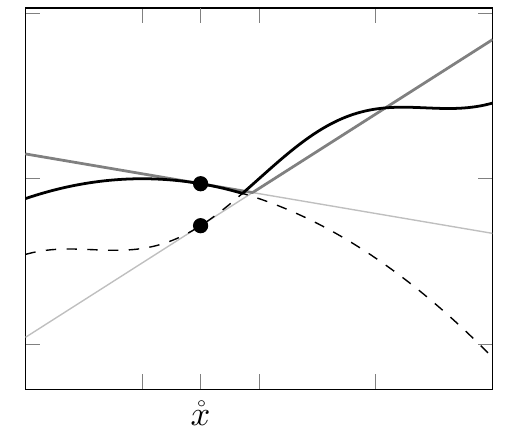}}
	  \hspace{.2cm}
	  \subfigure{\includegraphics{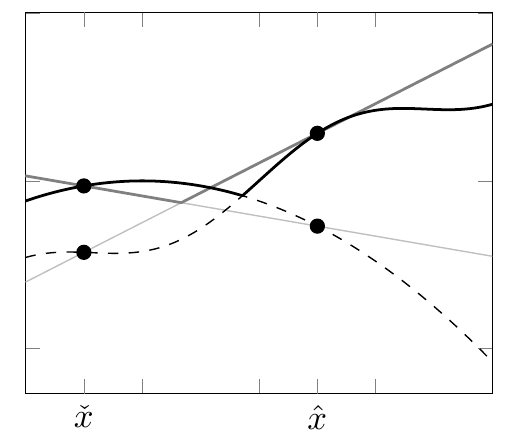}}
	  \caption{visualization of tangent and secant mode linearizations}\label{Fig:Modes}
	\end{figure}\\
	In order to utilize AD for the algorithmic computation of the secant approximation
	in \eqref{Eqn:SecantApproximation} we observe that in \eqref{Eqn:TangentLinearization}
	the intermediate values can be seen as functions evaluated at the unique reference
	point $\mathring x$, with $\mathring v_i = v_i(\mathring x)$. Now this reference point is the midpoint \(\rx=(\cx+\hx)/2\) and 
	the intermediate values are
	\begin{equation*}
		\mathring v_i = (\check v_i + \hat v_i)/2\; , \mbox{  with  }\; , \check v_i = v_i(\check x), \hat v_i = v_i(\hat x).
	\end{equation*}
	Replacing $v_i$ with $\mathring v_i$ in \eqref{Eqn:TangentLinearization} we observe
	that the first and second line are the same for the secant linearization.
	The third line has to be changed slightly, since the tangent slope $\mathring{c}_{ij}$
	has to be replaced by the secant slope
	\begin{equation*}
		\mathring c_{ij} = \left\lbrace\begin{array}{r l} \phi'_i(\mathring v_j) & \mbox{ if }\check v_j = \hat v_j\\ \dfrac{\hat v_i - \check v_i}{\hat v_j - \check v_j} &\mbox{otherwise}\end{array}\right.
	\end{equation*}
	The last rule stays unchanged except that now $\mathring v_i = (\check v_i + \hat v_i)/2 = (|\check v_j|+|\hat v_j|)/2$.
	Note that, if $\check x = \hat x$, we obtain $\Delta F(\rx,\Delta x) = \Delta F(\check x, \hat x;\Delta x)$.
	A complete discussion on this implementation topic can be found in \cite[Sec. 7]{mother}. Additionally, a division-free implementation and thus numerically stable implementation is discussed in \cite[Section 6]{NewtonPL}.

In contrast to the presentation in our previous papers we will now also use the nonincremental forms
 $$ \lozenge_{\rx}  F(x)     \; \equiv \;  F(\rx) + \Delta F(\rx; x-\rx )  $$
 and  
  $$ \lozenge_{\cx}^{\hx}  F(x)     \; \equiv \;  \tfrac{1}{2} ( F(\cx) + F(\hx))  + \Delta F(\cx, \hx; x-\rx )    $$

Hereafter we will denote by \(\lVert\cdot\rVert\equiv\lVert\cdot\rVert_\infty\) the infinity norm. Due to norm equivalence in finite dimensional spaces all inequalities to be derived take the same form in other norms, provided the constants are adjusted accordingly.

Moreover, we will frequently use the central Proposition 4.2 from \cite{NewtonPL}, which, in essence, states the following: 
\begin{itemize}
\item We can approximate piecewise differentiable functions with second order or bilinear error by piecewise linear tangent or secant models, respectively. 
\item Said models are globally Lipschitz continuous with the local Lipschitz constant \(\beta_F\) of the underlying function. 
\item The piecewise linear models are Lipschitz continuous with respect to their development points \(\rx\) or \(\cx,\hx\), respectively, and explicit bounds are given. 
\end{itemize}
 
\section{Generalized trapezoidal rule}
 
For the generalization of the trapezoidal rule, the linearization mode of choice is the secant mode. The construction largely follows the classical case (see, e.g. \cite{nummath}). We know that by the fundamental theorem of calculus it holds 
\[x(h)-x(0)=\int_0^h F(x(t))\mathrm dt\; .\]
Define \(\hx:=x(h)\) and \(\cx:=x(0)\). We then have: 
\[\hx-\cx=\int_0^h F(x(t))\mathrm dt=h\int_{-\frac12}^{\frac12}F\left(x\left(\tfrac h2+\tau h\right)\right)\mathrm d\tau\; .\]
Approximating \(x(t)\) by the secant
\[h\int_{-\frac12}^{\frac12}F\left(x\left(\tfrac h2+\tau h\right)\right)\mathrm d\tau=h\int_{-\frac12}^{\frac12}F\left(\cx\left(\tfrac12-\tau\right)+\hx\left(\tfrac12+\tau\right)\right)
  \mathrm d\tau+\mathcal O(h^3)\]
  and the latter expression, in contrast to the construction of the classical trapezoidal rule, by its piecewise linearization, where equality holds as consequence of \cite[Prop. 4.2]{NewtonPL}
 
\[h\int_{-\frac12}^{\frac12}F\left(\cx\left(\tfrac12-\tau\right)+\hx\left(\tfrac12+\tau\right)\right)\mathrm d\tau=h\int_{-\frac12}^{\frac12}\mathring F+\Delta F(\cx,\hx;\tau(\hx-\cx))\mathrm d\tau + 
  \mathcal O(h^3)\; \]
we arrive at the following defining equation, which was introduced by Griewank in \cite{mother}:
\begin{equation}
\label{gen_trap}
\hat x-\check x=h\integ{\left[\mathring F+\Delta F(\check x,\hat x;(\hat x-\check x)t)\right]}\; ,
\end{equation}
where
 \(\check x\) is the current and \(\hat x\) the next point on the numerical trajectory and  \(\mathring F=\frac12\left[F(\hx)+F(\cx)\right]\)  \cite[S.~21]{mother}. 
 
 This construction offers two major benefits: First, it has the desired property of having a third order local truncation error even when integrating through a kink. Second, it is consistent with the classical trapezoidal rule in the sense that in case of a smooth function $F$ the generalized formula reduces to the classical one and thus represents a proper extension of the classical theory \cite[S.~21]{mother}. It then holds: 
\begin{align*}
\hat x-\check x&=h\integ{\left[\mathring F+\Delta F(\hat x,\check x;(\hat x-\check x)t)\right]} \\
&=h\integ{\left[\mathring F+\left[F(\hat x)-F(\check x)\right]t\right]} = h \mathring F\; .
\end{align*}
To simplify the equality (\ref{gen_trap}) (which also yields a simplification of next sections' convergence proof) we assume without loss of generality that our current point is the initial value  \(\cx:=x(0)=0\). We then get $\check x= 0, x:=\hat x$ and \(\mathring x = x/2\) and the above formula simplifies as follows \cite[S.~22]{mother}:  
\begin{align*}
x&=h\integ{\left[\mathring F+\Delta F\left(0,x;xt\right)\right]} \\
&=\frac h2 \left[F(0)+F(x)\right]+h\int^\frac12_{\frac12}\Delta F\left(0,x;xt\right)\mathrm dt=:hG(x)\; .
\end{align*}
A generalized midpoint rule can be derived in an analogous fashion (cf.\;\cite[Section 5.2]{mother}): 
\[\hat x-\check x=h\integ{\left[F(\rx)+\Delta F(\rx;(\hat x-\check x)t)\right]}\, .\]
However, while many of the results derived for the trapezoidal rule hold for the midpoint rule as well, the latter does have certain practical disadvantages. For example the error estimator developed below cannot be applied to it. We thus limit our attention to the trapezoidal rule.

\subsection*{Convergence results}

It was shown in \cite{mother} that the generalized midpoint rule has convergence order two (with respect to the step size $h$). We will now prove the analogous result for the generalized trapezoidal rule. Note that the arguments we employ for this are mostly similar to those used in the aforementioned reference. 

\begin{theorem}[Griewank, Proposition 5.2]
Suppose, $F$ is piecewise composite differentiable in the sense defined above and Lipschitz continuous in an open neighborhood 
 \(\mathcal D\) of the origin \(\check x=0\). Then there is a bound 
\(\bar h>0\) for the step size, such that for all \(h<\bar h\) the function \(hG(x)\) maps some closed ball \(B_\rho(0)\subset\mathcal D\), \(\rho>0\) into itself and is contractive. Moreover, the unique fixed point \(x_h\in B_\rho(0)\) satisfies the equality 
 \[x_h-x(h)=\mathcal O(h^3)\; ,\] where $x(t)$ is a solution of the differential equation $\dot x(t)=F(x(t))$ with $x(0)=0$.
\end{theorem}

\begin{proof}
 \(F\) is, by assumption, piecewise composite differentiable and thus locally Lipschitz continuous. Moreover, by \cite[Prop. 4.2]{NewtonPL} we know that the piecewise linearization is Lipschitz continuous. Consequently, with \cite[Prop. 4.2]{NewtonPL} there exists a ball \(\mathcal B_\rho(0)\) about the base point 
 \(x_0=0\), such that for all \(x\in\mathcal B_\rho(0)\) there exists a \(\beta_F>0\), such that it holds 
 \[\norm{F(x)-F(0)}\leq \beta_F\rho \quad\text{ as well as }\quad \norm{\lozenge_0^x F(0)-\lozenge_0^x F(x)}\leq \beta_F\rho\; .\]
From \cite[Prop. 4.2 (iii)]{NewtonPL} we can also derive that:
\[\norm{\lozenge_0^{x_1}F(x)-\lozenge_0^{x_2}F(x)}\leq\gamma_F\norm{x_1-x_2}\norm{x} \]
for \(\gamma_F\) as defined there. Employing this we get: 
 \begin{align*}
  &\norm{G(\tilde x)-G(x)} = \norm{\int_{-\tfrac 12}^{\tfrac 12}\lozenge_0^{\tilde x}F(\tfrac{\tilde x}2+t\tilde x)-\lozenge_0^{x}F(\tfrac{x}2+tx)}\mathrm dt\\ 
  &\leq \int_{-\tfrac 12}^{\tfrac 12} \norm{\lozenge_0^{\tilde x}F(\tfrac{\tilde x}2+t\tilde x)-\lozenge_0^{x}F(\tfrac{x}2+tx)}\mathrm dt \\
  &\leq \int_{-\tfrac 12}^{\tfrac 12} \norm{\lozenge_0^{\tilde x}F(\tfrac{\tilde x}2+t\tilde x)-\lozenge_0^x F(\tfrac{\tilde x}2+t\tilde x)} + \norm{\lozenge_0^x F(\tfrac{\tilde x}2+t\tilde x)-\lozenge_0^{x}F(\tfrac{x}2+tx)}\mathrm dt \\
  &= \int_{-\tfrac 12}^{\tfrac 12} \gamma_F\norm{\tilde x-x}\norm{\tfrac{\tilde x}2+t\tilde x} + \beta_F\norm{\frac{\tilde x}2+t\tilde x-\left(\tfrac x2+tx\right)}\mathrm dt \\
  &= \norm{\tilde x-x}(\gamma_F\norm{\tilde x}+\beta_F)\int_{-\tfrac 12}^{\tfrac 12}\left|t+\tfrac 12\right|\mathrm dt = \tfrac 12 (\gamma_F\norm{\tilde x} + \beta_F)\norm{\tilde x-x} \\
  &\leq \tfrac 12(\gamma_F\rho+\beta_F)\norm{\tilde x-x} =: \tilde\beta\norm{\tilde x-x}
 \end{align*}
where \(\tilde \beta=\frac12(\gamma_F\rho+\beta_F)\) is a Lipschitz constant for \(G(x)\). Consequently, \(h\tilde \beta\) is a Lipschitz constant for  \(hG(x)\). Therefore \(hG(x)\) is a contraction if we can ensure that \(h\tilde \beta<1\) is the case for \(h\) sufficiently small. Since we know it holds 
\[G(0)=\frac12\left[F(0)+F(0)\right]+\integ{\incf{0,0;0}}=F(0)\]
 and 
 \[\norm{hG(x)-hG(0)}\geq\norm{hG(x)}-\norm{hG(0)},\] 
 it follows that
 \[\norm{hG(x)}\leq \norm{hG(0)}+h\tilde \beta\norm{x}=h\norm{F(0)}+h\tilde \beta\norm{x}<\rho\] 
 for \(h\) 
sufficiently small. Hence, \(hG(x)\) maps the ball \(\mathcal B_\rho(0)\) into itself. With this knowledge we can apply Banachs fixed point theorem and get that the fixed point iteration \(hG(x)\) has a unique fixed point \(x_h\in\mathcal B_\rho(0)\). We now consider the trajectory \(x(t)\) of the exact solution of the differential equation, which is in \(C^{1,1}\), since \(F\) is Lipschitz continuous. We approximate the latter with the secant \((t+0.5)x(h)\)
for \(-0.5\leq t\leq0.5\). This corresponds to a polynomial interpolation with a first order polynomial. Using the following auxiliary  we can thus estimate the interpolation error. Define
\begin{align*} 
g(t)&:=x((t+0.5)h)-(t+0.5)x(h)-\frac{\frac14-t^2}{\frac14-\tau^2}\left[x((\tau+0.5)h)-(\tau+0.5)x(h)\right]\\
&=:\delta(t)-\frac{\frac14-t^2}{\frac14-\tau^2}\delta(\tau)\; ,
\end{align*}
for a \(\tau\in\left[-\frac12,\frac12\right]\). By construction this function is in \(C^{1,1}\) and has the three roots \(-\frac12,\frac12,\tau\). Hence its derivative  
\[g^\prime(t)=\delta^\prime(t)-\frac{-2t}{\frac14-\tau^2}\delta(\tau)\]
has two roots \(t_1,t_2\). For these it holds: 
\[\delta^\prime(t_1)=\frac{-2t_1}{\frac14-\tau^2}\delta(\tau)\]
and 
\[\delta^\prime(t_2)=\frac{-2t_2}{\frac14-\tau^2}\delta(\tau)\; .\]
Then we have 
\[\delta^\prime(t_2)-\delta^\prime(t_1)=\frac{2\delta(\tau)}{\frac14-\tau^2}(t_1-t_2).\]
We know about \(\delta^\prime(t)\) that \(\delta^\prime(t)=x^\prime((t+0.5)h)h-x(h)\) and consequently
\begin{align*}
 \norm{\delta^\prime(t_2)-\delta^\prime(t_1)}&=\norm{F((t_2+0.5)h)h-F((t_1+0.5)h)h} \\
 &= h\norm{F((t_1+0.5)h)-F((t_2+0.5)h)}\leq \beta_F h^2\norm{t_1-t_2}
\end{align*}
since \(F\) is Lipschitz continuous. This yields
\[\norm{\frac{\delta^\prime(t_2)-\delta^\prime(t_1)}{t_1-t_2}}=\norm{\frac{2\delta(\tau)}{\frac14-\tau^2}}\leq\beta_F h^2\]
and accordingly
\[\norm{\delta(\tau)}\leq\frac12 \beta_F h^2\left(\frac14-\tau^2\right)\; .\]
Hence, the following inequality holds: 
\[\norm{x((t+0.5)h)-(t+0.5)x(h)}\leq\frac{\beta_F}2\left(\frac14-t^2\right)h^2\; ,\]
where \(t\in(-\frac12,\frac12)\). Moreover, by the fundamental theorem of calculus, we know that  
\[\norm{x(h)-h\integ{F(x(t+0.5)h)}}=\norm{x(h)-x(h)+x(0)}=0\; .\] This gives 
\begin{align*}
0&=\norm{x(h)-h\integ{F(x((t+0.5)h))}} \\
&=\norm{x(h)-h\integ{F((t+0.5)x(h))}}-\mathcal O(h^3)\; ,
\end{align*}
since it holds 
\begin{align*}
 &\norm{x(h)-h\integ{F(x((t+0.5)h))}}-\norm{x(h)-h\integ{F((t+0.5)x(h))}} \\
 &\leq \norm{x(h)-h\integ{F(x((t+0.5)h))}-\left[x(h)-h\integ{F((t+0.5)x(h))}\right]} \\
 &= \norm{h\integ{F(x((t+0.5)h))-F((t+0.5)x(h))}} \\
 &\leq h\integ{\norm{F(x((t+0.5)h))-F((t+0.5)x(h))}} \\
 &\leq h\beta_F\integ{\norm{x((+0.5)h)-(t+0.5)x(h)}} \leq h\beta_F\integ{\frac{\beta_F}2\left(\frac14-t^2\right)h^2} \\
 &=h \frac{\beta_F^2 h^2}{12} \in \mathcal O(h^3)\; .
\end{align*}
But, reapplying \cite[Prop. 4.2 (ii)]{NewtonPL}, we also get 
\begin{align*}
 &\norm{x(h)-hG(x(h))}-\norm{x(h)-\integ{F((t+0.5)x(h))}} \\ 
 &\leq h\norm{G(x(h))-\integ{F((t+0.5)x(h))}} \\
 &= h\norm{\frac12\left[F(x(h))-F(0)\right]+\integ{\incf{0,x(h);tx(h)}-F((t+0.5)x(h))}} \\
 &\leq h\integ{\norm{F((t+0.5)x(h))-\frac12\left[F(0)-F(x(h))\right]-\incf{0,x(h);tx(h)}}} \\
 &=: h\integ{\norm{F(\tilde x)-\mathring F-\incf{0,x(h);\tilde x-\mathring x}}} \\
 &\leq h\tfrac{1}{2}\gamma_F\integ{\norm{\tilde x-0}\norm{\tilde x-x(h)}} \\ 
 &= h\tfrac{1}{2}\gamma_F\integ{\norm{(t+0.5)x(h)}\norm{(t-0.5)x(h)}} \\
 &= h\tfrac{1}{2}\gamma_F\integ{\betrag{t+0.5}\betrag{t-0.5}\norm{x(h)}^2} \\
 &= h\gamma_F\frac{\norm{x(h)}^2}{12} \leq h\gamma_F\frac{h^2}{12} \in \mathcal O(h^3)\; .
\end{align*}
 The last inequality holds, since \(x(h)\in\mathcal B_\rho(0)\). Now the mapping \(x-hG(x)\) has a locally Lipschitz continuous inverse. As we know that \(h\beta<1\) for \(h\) sufficiently small, it follows from Banachs fixed point theorem that the operator \(I+hG\) consisting of the identity and a Lipschitz continuous function with constant smaller 1 has a Lipschitz continuous inverse function. Since 
 \[x(h)-hG(x(h))-\left[x_h-hG(x_h)\right]=\mathcal{O}(h^3)-0,\]
 we can conclude that for the images of the inverse function it holds 
 \[\norm{x(h)-x_h}\leq \tilde \beta\norm{\left[x(h)-hG(x(h))\right]-\left[x_h-hG(x_h)\right]}=\mathcal{O}(h^3)\]
 for a \(\tilde \beta>0\), as \(x_h=\left(I-hG\right)^{-1}(0)\) and \(x(h)=\left(I-hG\right)^{-1}(s(h))\), with \(s(h)\in\mathcal O(h^3)\).  
\end{proof} 

\section{Properties of the presented methods}

We say that the solution of an ODE has finite transition, if the solution trajectory intersects the nondifferentiabilities of the right hand side in at most a finite point set. This has an impact on the overall performance of the considered methods. If said criterion is violated the classical trapezoidal rule drops to first order global convergence. This is not the case for the generalized rule, because the third order local truncation error is maintained even on nondifferentiabilities. This leads to the improved convergence order of the generalized method on these problems as illustrated in the table below. 

\begin{minipage}[c]{\textwidth}
\centering
\vspace{10pt}
\begin{tabular}{| l | c | c | c |}
\hline
Without Finite Trans. & on smooth parts & on kinks & globally  \\
\hline \hline
Classical Rule & $\mathcal{O}(h^3)$ & $\mathcal{O}(h^2)$ & $\mathcal{O}(h)$\\ 
Generalized Rule & $\mathcal{O}(h^3)$ & $\mathcal{O}(h^3)$ & $\mathcal{O}(h^2)$\\
\hline
\end{tabular}
\vspace{10pt}
\end{minipage}
Of course the above condition is met in most continuous examples. However even in this case we can still expect a gain as described below. Note that finite transition does not require the trajectory to be transversal to the sets of kink locations. The latter, stronger property is required for efficient event handling by computing the roots of switching functions \cite{events}. They must be singular a tangential transition points. 

\subsection*{Richardson Extrapolation}
In the following we assume that all solutions have finite transition. It is well known that the local truncation error for the Romberg extrapolation is of order five for sufficiently smooth function \(F\). Hence its maximal order that can be expected for the global error is four. However if said \(F\) is only piecewise differentiable, then the order collapses on the kinks, posing an upper bound for the global error. Thus the overall global accuracy in the case of Romberg extrapolation is determined by the respective behavior of the investigated method on the nondifferentiablities of \(F\). Here the generalized method gains its significance, as its accuracy only collapses to an error of order three, as opposed to order two for the classical method. This is the error that the respective methods would achieve without extrapolation. It is lower for the classical method, because the linear approximation used in its construction is only of first order on the kinks as opposed to second order for the piecewise 
linear approximation of the generalized method (also see the table below). 

\begin{minipage}[c]{\textwidth}
\centering
\vspace{10pt}
\begin{tabular}{| l | c | c | c |}
\hline
With Finite Transition & on smooth parts & on kinks & globally  \\
\hline \hline
Classical Rule & $\mathcal{O}(h^3)$ & $\mathcal{O}(h^2)$ & $\mathcal{O}(h^2)$\\ 
Generalized Rule & $\mathcal{O}(h^3)$ & $\mathcal{O}(h^3)$ & $\mathcal{O}(h^2)$\\
\hline
Class. w/ Romberg & $\mathcal{O}(h^5)$ & $\mathcal{O}(h^2)$ & $\mathcal{O}(h^2)$\\
Gen. w/ Romberg & $\mathcal{O}(h^5)$ & $\mathcal{O}(h^3)$ & $\mathcal{O}(h^3)$ \\
\hline
\end{tabular}
\vspace{10pt}
\end{minipage}
One might have hoped that extrapolation would also yield a local of \(\mathcal O(h^5)\) on kinks. The following, easily verifiable example shows that this is not the case.  

\begin{lemma}
The analytical solution to the nonsmooth ODE
\vspace{-5pt}
\begin{align*}
\dot x &= a\lvert x\rvert + b x + 1 = \begin{cases}(b-a)x+1 & x\geq 0\\ (b+a)x+1 & else\end{cases} 
\end{align*}
is 
\begin{align*}
  x_+(t) &= \frac{e^{t(b-a)}-1}{b-a} \quad\text{and}\quad x_-(t) = \frac{e^{t(b+a)}-1}{b+a} \,.
\end{align*}
\end{lemma}
If now we set \(a=2.25\) and \(b=\)-1.25, we can calculate the local truncation error for a single step over the kink \(x=0\). For the classical method it amounts to  \(\frac{27h^2}{64}+\frac{677h^3}{1536}+\mathcal{O}(h^4)\) and for the generalized method to \(\frac{139h^3}{3072}+\mathcal{O}(h^4) \). Using Romberg extrapolation does not improve the order of the error. In this case we get \(\frac{3h^2}{64}+\frac{51h^3}{512}+\mathcal{O}(h^4)\) for the classical method and \(\frac{9h^3}{1024}+\mathcal{O}(h^4)\) for the generalized method. Consequently, as opposed to the classical case, Romberg extrapolation only yields a third order global convergence instead of the usual \(\mathcal O(h^4)\). 

\subsection*{Remark on Geometric Integration}

Among ODE integration methods, those which allow for the preservation of certain geometric properties, especially energy preservation and symplecticness, play an important role in current research. The presented method is part of this category for piecewise linear Hamiltonian systems. To show this, we note that the piecewise linearization of a piecewise linear function is the function itself. Hence, for a piecewise linear right hand side, the formula for the generalized trapezoidal rule simplifies as follows:

\[ \hx - \cx = h \int_{-\frac12}^{\frac12} F\left(\frac{\hx + \cx}{2} + t(\hx - \cx)\right) \mathrm dt \,.\]
In \cite{quispel} the concept of a so-called average vector field method (AVF) is introduced in terms of the following formula:

\[ \hx - \cx = h \int_0^1 F((1-s)\cx + s\hx)\mathrm ds\]
An AVF is energy conserving on Hamiltonian systems. However, since for this method an exact integration of the right hand side is necessary, there only exists a straightforward implementation for linear systems. But the generalized trapezoidal rule performs an exact integration of the right hand side in case \(F\) is piecewise linear. Thus it is energy preserving on piecewise linear Hamiltonians. 

\section{Continuous Output}
It is well known that for smooth systems the classical trapezoidal rule gives us a continuous output function. On a single integration step, this takes the form of a quadratic polynomial \(p\colon[0,h]\to\R^n\) with a third order approximation error. It is given by
\[p(t) = \cx + \int_0^t F(\cx) + \tfrac \tau h(F(\hx)-F(\cx))\ \mathrm d\tau\quad \text{for }t\in[0,h]\,.\]
This polynomial is tangential to the numerical trajectory in the sense that its values at \(t=0\) and \(t=h\) equal \(\cx\) and \(\hx\), respectively, and its slope matches the vector field of the numerical solution in said points. The integral can be evaluated and gives an explicit formula
\[ p(t) = \frac{F(\hx)-F(\cx)}{2h}t^2 + F(\cx)t+\cx\,.\]
However, in the case of a step through a kink, the linear approximation of \(F\) used in the trapezoidal rule is only a first order approximation. Thus the above polynomial is only a second order approximation of the trajectory, which is not sufficient for certain applications like the construction of an error estimator that we want to pursue below. Fortunately the generalized trapezoidal rule allows for the construction of such a polynomial with the desired properties for nonsmooth functions. It is derived in an analogous way and takes the form: 
\[p(t) = \cx + \int_0^t \lozenge_{\cx}^{\hx} F(\cx+\tfrac \tau h(\hx-\cx))\mathrm d\tau\quad\text{for }t\in[0,h]\,.\]
Of course it is now a piecewise quadratic function which consists of some \(p_i\) with
\[p_i(t) = \left. p\right|_{{[h\tau_i,h\tau_{i+1}]}}(h\tau_i+t)=a_i t^2+b_i t +c_i\,.\]
Here \(h\tau_i\) is the length of the step to the next kink, or to \(\hx\) if there is no kink. These values are contained in the piecewise linearization of \(F\) and are calculated during the integration. We also need the intermediate values \(\cx_{\to,i}\) of the numerical trajectory at the kinks for which it holds: 
\[\cx_{\to,i} = \cx + h\int_0^{\tau_i}\lozenge_{\cx}^{\hx}F(\cx+t(\hx-\cx))\mathrm dt\,.\]
They are already calculated as well, since the integral from \(\cx\) to \(\hx\) is simply the sum of the values. If there are \(k\) kinks in the observed time step, we have \(\tau_0=0\) and \(\tau_{k+1}=1\), such that \(\cx_{\to,k+1}=\hx\). Consequently the derivatives \(\dot \cx_{\to,i}\) are given by \(\dot \cx_{\to,i}=F(\cx_{\to,i})\). We can now derive the coefficients of said interpolants \(p_i\), since we know that: 
\begin{align*}p_i(0) &= \cx_{\to,i} \quad\Longrightarrow\quad c_i = \cx_{\to,i} \,,  \\
p_i^\prime(0) &= \dot \cx_{\to,i} \quad\Longrightarrow\quad b_i=\dot \cx_{\to,i} \,, \\
p_i^\prime(h(\tau_{i+1}-\tau_i)) &= \dot \cx_{\to,i+1} \vspace{.5cm} \Longrightarrow\quad a_i=\frac{\dot \cx_{\to,i+1}-\dot \cx_{\to,i}}{2h(\tau_{i+1}-\tau_i)} \,.
\end{align*}
This means \(p\) is given by
\[p\left(\tilde t\right)=p(h\tau_i+t)=p_i(t)=\frac{\dot \cx_{\to,i+1}-\dot \cx_{\to,i}}{2h(\tau_{i+1}-\tau_i)}t^2 + \dot \cx_{\to,i}t + \cx_{\to,i} \quad\text{for }\tilde t\in[h\tau_i,h\tau_{i+1}],i\in\{0,\dots,k\} \,.\]
This polynomial has correct values and slopes at \(\cx, \hx\) and all the kinks. Consequently it is a third order approximation on the intervals between consecutive kinks or \(\cx\) or \(\hx\), respectively and thus everywhere. It is depicted in figure \ref{fig:interpol}. 
\begin{figure}
\centering
\includegraphics[width=0.7\textwidth]{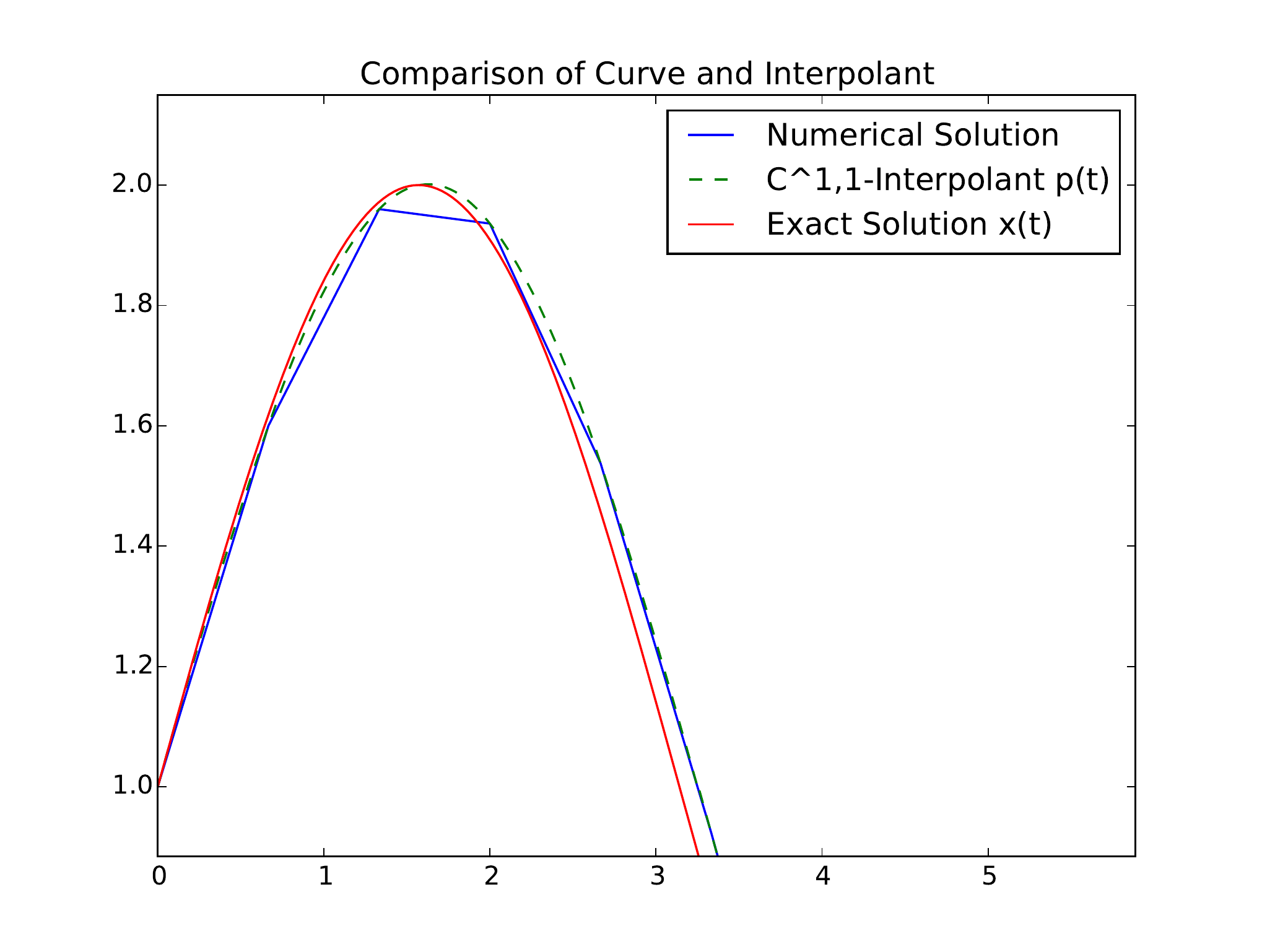}
\caption{Continuous Output, Rolling Stone Example, \(h=\tfrac 23\)}
\label{fig:interpol}
\end{figure}

\section{Error Estimation and Time-Stepping}
In this section we will construct an error estimator for the local truncation error of the generalized trapezoidal rule. 

\begin{proposition}[Error estimator]
The error estimator has the form
\begin{align*}
&\norm{x(h) - x_h} \leq \tfrac 1{12}h\gamma_F\norm{\hx-\cx}^2\\
+&\ \beta_F\sum_{i=0}^k\int_0^{h(\tau_{i+1}-\tau_i)}\norm{\frac{\dot \cx_{\to,i+1}-\dot \cx_{\to,i}}{2h(\tau_{i+1}-\tau_i)}t^2 + \left(\dot \cx_{\to,i} - \frac{\hx-\cx}{h}\right)t + \cx_{\to,i} - \cx - \tau_i(\hx-\cx)}\mathrm dt \,.
\end{align*}
\label{error_est_formula}
\end{proposition}
Note that for the evaluation of this formula the integral over the absolute value of a quadratic function has to be calculated. In general, this necessitates the computation of the roots of the polynomial. 

%

We start the derivation of the error estimator by splitting up the truncation error into two components:  

\begin{align*}
\norm{x(h) - x_h} &= \norm{\int_0^h F(x(t))\mathrm dt - \int_0^h \lozenge_{\cx}^{\hx} F(\cx+\tfrac th(\hx-\cx))\mathrm dt } \\
&\leq \int_0^h \norm{ F(x(t)) -  \lozenge_{\cx}^{\hx} F(\cx+\tfrac th(\hx-\cx))}\mathrm dt \\
&\leq \int_0^h \norm{ F(x(t)) - F(\cx+\tfrac th(\hx-\cx))} \\ &\hspace{.68cm} + \norm{ F(\cx+\tfrac th(\hx-\cx)) - \lozenge_{\cx}^{\hx} F(\cx+\tfrac th(\hx-\cx))}\mathrm dt \,.
\end{align*}
The left term of the last expression can be bounded, using the Lipschitz constant \(\beta_F\) calculated above: 
\[\int_0^h \norm{ F(x(t)) - F(\cx+\tfrac th(\hx-\cx))}\mathrm dt \leq \beta_F \int_0^h \norm{ x(t) - \cx-\tfrac th(\hx-\cx)}\mathrm dt \,.\]
Since the analytical solution trajectory is unknown we approximate it with the piecewise quadratic interpolation polynomial constructed above, whose approximation is of order \(\mathcal O(h^3)\), which does not decrease the overall approximation error of order two. 
\begin{align*}
\beta_F \int_0^h \norm{ x(t) - \cx-\tfrac th(\hx-\cx)}\mathrm dt &= \beta_F \int_0^h \norm{ p(t) - \cx-\tfrac th(\hx-\cx)+\mathcal{O}(h^3)}\mathrm dt \\ &=\beta_F \int_0^h \norm{ p(t) - \cx-\tfrac th(\hx-\cx)}\mathrm dt + \mathcal{O}(h^4)\,.
\end{align*}
This integral can be split up into the sum of the integrals from kink to kink:
\begin{align*}
&\ \beta_F \int_0^h \norm{ p(t) - \cx-\tfrac th(\hx-\cx)}\mathrm dt = \beta_F \sum_{i=0}^k \int_0^{h(\tau_{i+1}-\tau_i)}\norm{p_i(t)-\cx-\tfrac{h\tau_i+t}h(\hx-\cx)}\mathrm dt \\ =\;& \beta_F\sum_{i=0}^k \int_0^{h(\tau_{i+1}-\tau_i)}\norm{\frac{\dot \cx_{\to,i+1}-\dot \cx_{\to,i}}{2h(\tau_{i+1}-\tau_i)}t^2 + \left(\dot \cx_{\to,i} - \frac{\hx-\cx}{h}\right)t + \cx_{\to,i} - \cx - \tau_i(\hx-\cx)}\mathrm dt\,,
\end{align*}
which is the first part of the error estimator. The second part can be bounded, using \cite[Prop. 4.2.]{NewtonPL}. 
\begin{align*}
&\int_0^h\norm{ F(\cx+\tfrac th(\hx-\cx)) - \lozenge_{\cx}^{\hx} F(\cx+\tfrac th(\hx-\cx))}\mathrm dt \\
\leq&\ \tfrac 12\gamma_F\int_0^h\norm{\tfrac th(\hx-\cx)}\norm{(\tfrac th - 1)(\hx-\cx)}\mathrm dt \\
=&\ \tfrac 12\gamma_F\norm{\hx-\cx}^2\int_0^h \left| \tfrac{t^2}{h^2}-\tfrac th\right| \mathrm dt = \tfrac 1{12}h\gamma_F\norm{\hx-\cx}^2\,.
\end{align*}
Hence, the overall error estimator has the form: 

With this formula in hand, a step size control can be implemented, just as in the classical case.

\section{Numerical Examples}
\subsection*{Rolling Stone}
	This example tracks a point moving frictionless on a convex surface
	representing an idealized rolling stone. 
	It can be considered as a harmonic oscillator provided the surface is parabolic.
	We modify this parabola by inserting a planar section in the interval $[-1,1]$.
	This yields the curve
	\begin{equation*}
		V(x) = \begin{cases}
		\ \tfrac 12(1+x)^2, &\hspace{8pt} x\leq-1 \\
		\ 0\,, & -1<x<1\;.\\
		\ \tfrac 12(1-x)^2, &\hspace{8pt}1\leq x
		\end{cases}
	\end{equation*}
The derivative of $V$ defining the acceleration $\ddot x$ of the mass is piecewise linear and given by
	\begin{equation*}
		-V'(x) = \min(\max(-1-x,0),1-x) = -x-|x-1|/2 + |x+1|/2
	\end{equation*}
	which yields the ODE $\ddot x = -V'(x)$.
	The analytic solution of the problem is $(2\pi+4)$-periodic and given by
	\begin{equation*}
		x(t) = \left\lbrace
				\begin{array}{r l}
					1 + \sin(t)	& 0 \leq t \leq \pi\\
					1 - (t - \pi)	& \pi \leq t < \pi + 2\\
					-1-\sin(2-t)	& \pi + 2\leq t < 2\pi +2\\
					t-3-2\pi	& 2\pi+2\leq t < 2\pi+4
				\end{array}\right.
	\end{equation*}
	In Figure \ref{FIG_Visualization} we depicted $V$ and $V'$ as well as the analytic solution
	of the ODE. The linear parts are drawn in gray.
	\begin{figure}[]
	\centering
	\subfigure{\includegraphics[height=170pt]{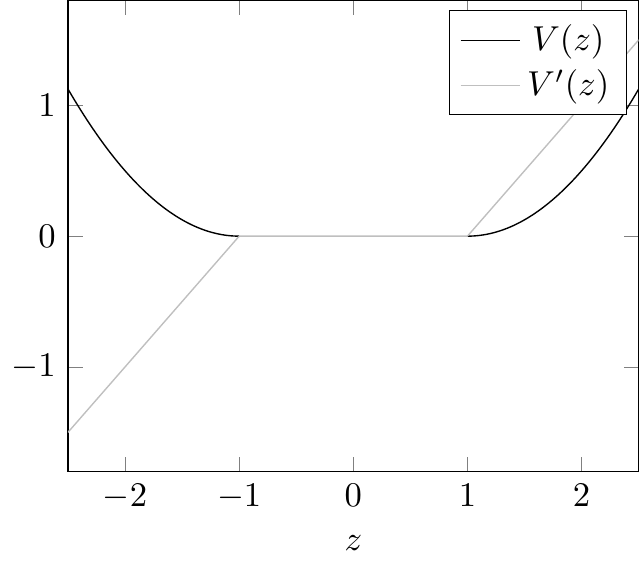}}
	\subfigure{\includegraphics[height=170pt]{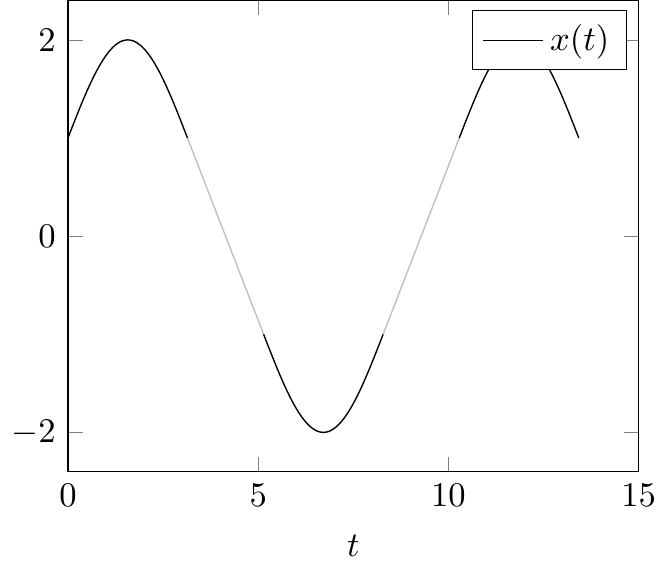}}
	\caption{Visualization of $V$ and the analytic solution of the ODE}\label{FIG_Visualization}
	\end{figure}\\
	From the second order ODE we obtain the first order system
	\begin{equation*}\label{RollingStone:ODESystem}
		\left(\begin{array}{c} \dot x_1\\ \dot x_2\end{array}\right)	= \left(\begin{array}{l} x_2\\-x_1-|x_1-1|/2 + |x_1+1|/2\end{array}\right) = F(x) \,.
	\end{equation*}
	We will consider the initial conditions $x_1(0)=1$ and $x_2(0)=1$. As predicted, we observe a global convergence order of two for both the classical and generalized method. However, it is clearly visible that the generalized method is more stable. This is a consequence of the generalized methods' greater accuracy on the kinks, which is also the reason that extrapolating yields an increased convergence order only for the generalized method. 
	
\begin{figure}
\centering
\includegraphics[width=0.6\textwidth]{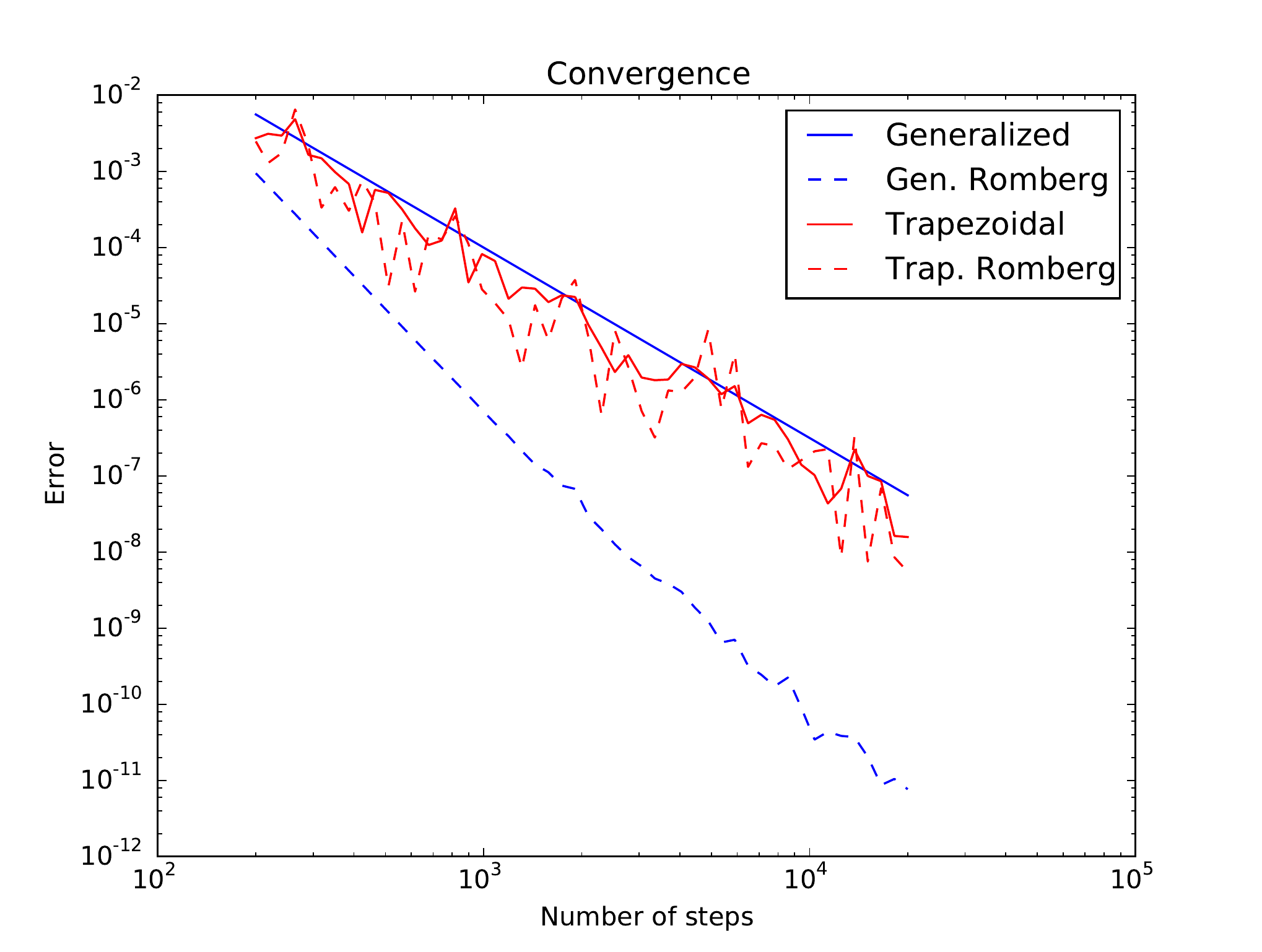}
\caption{Convergence Rolling Stone}
\label{conv_stones}
\end{figure}	

Due to the simplicity of the example adaptive time-stepping does not improve the solution significantly. We thus omit its investigation for the latter. But as a frictionless mechanical system it is energy preserving. As a piecewise linear Hamiltonian system it fulfills the requirements for the generalized method to correctly preserve this energy. Accordingly, the total energy
\[V(x(t))+\tfrac 12\dot x(t)^2\]
must stay constant at its initial value \(V(x_0)+\tfrac 12 y_0^2 = \tfrac 12\). Hence, we consider the total energy variation over all \(N\) time steps using the formula
\[\left[\sum_{i=1}^N\left(V(x_{1,i})+\tfrac 12(x_{2,i})^2 -\tfrac 12\right)^2\right]^{\tfrac 12}\,.	\]
\begin{figure}
\centering
\subfigure[Energy Error]{\includegraphics[height=150pt]{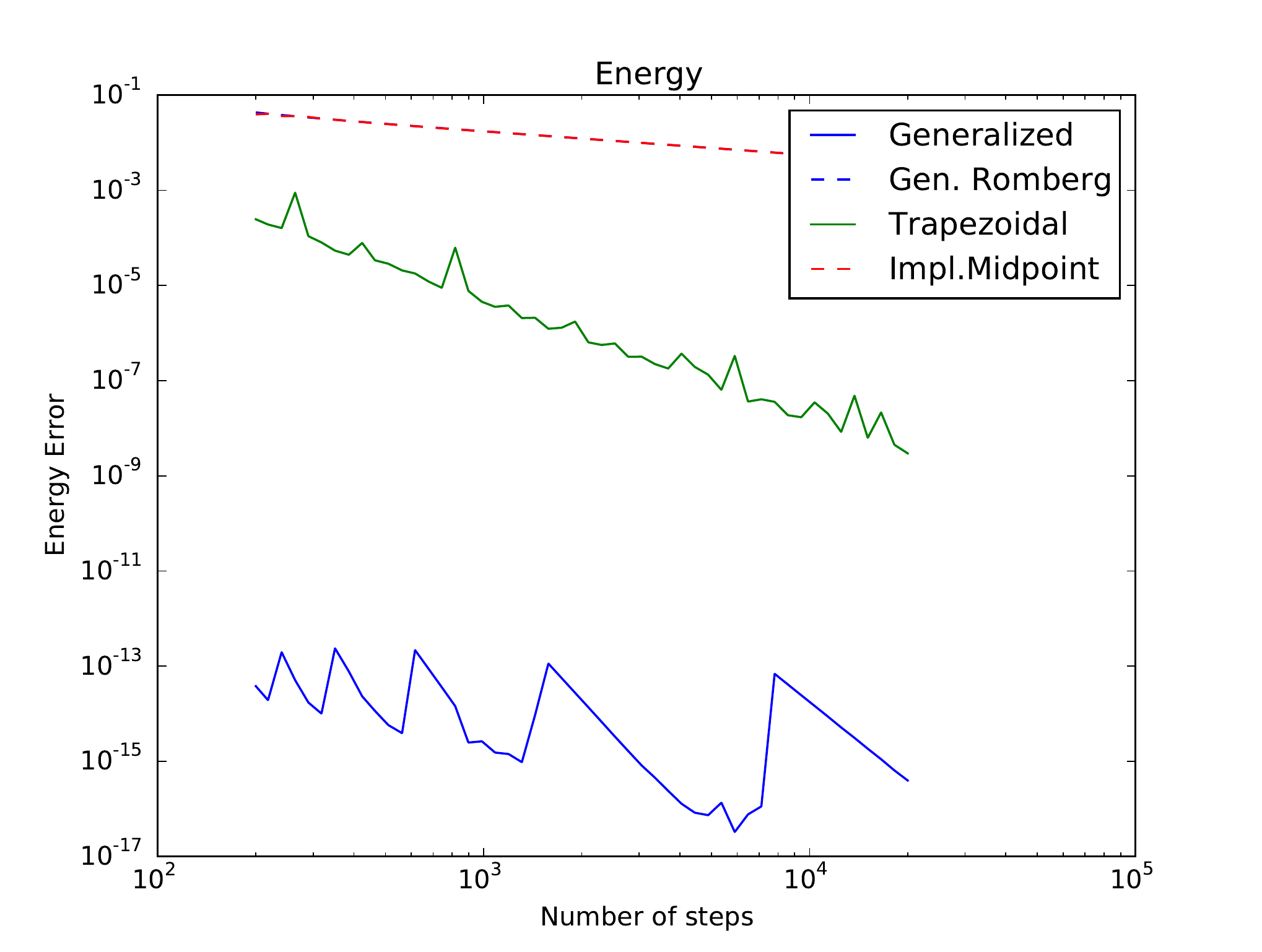}}
\hspace{5pt}
\subfigure[Energy Loss over Time]{\includegraphics[height=150pt]{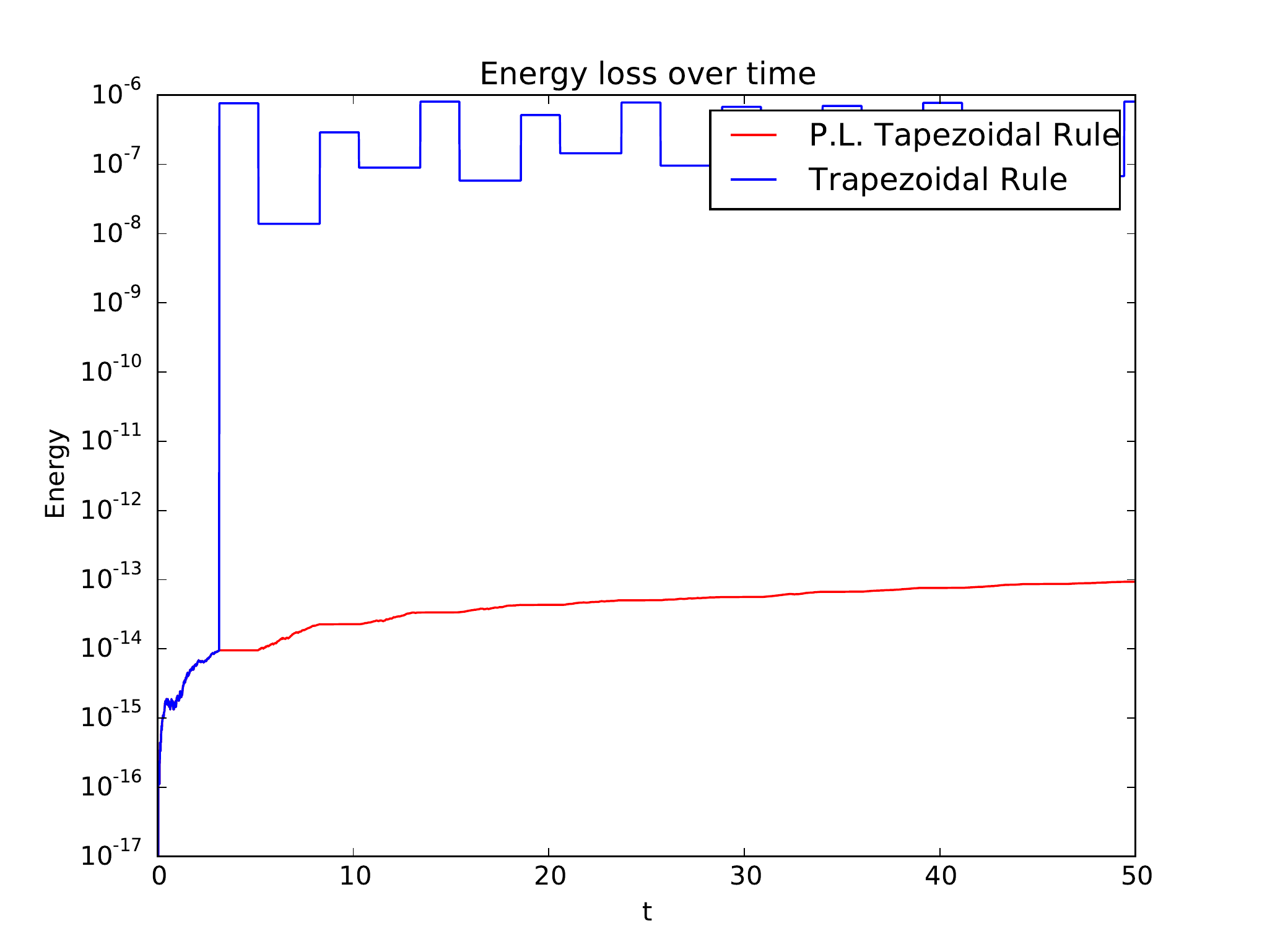}}
\caption{Energy Preservation}
\label{NRG_stones}
\end{figure}
We observe in Figure \ref{NRG_stones} (a) that the generalized method performs in accordance with this expectation, whereas the classical method is not energy preserving on the kinks, as can be seen in Figure \ref{NRG_stones} (b), and consequently does not preserve the systems' energy globally, either. 

\subsection*{Diode LC-Circuit}
The second example is more alike problems arising in actual applications: We take a simple LC-circuit and replace the 
	resistor with a diode providing an element which causes a nondifferentiable impact in the equations
	describing the system. Figure \ref{FIG_CircuitDiagram} depicts the circuit.
	\begin{figure}[ht!]
		\centering
		\includegraphics[]{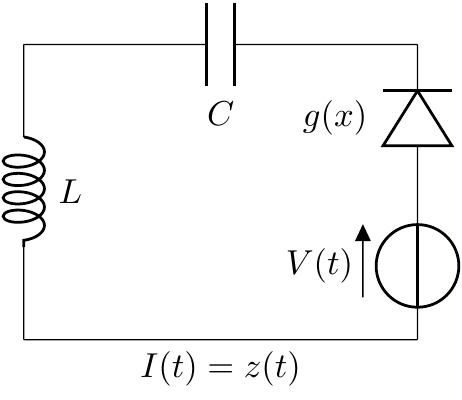}
		\caption{Circuit Diagram}\label{FIG_CircuitDiagram}
	\end{figure}
	It is modeled by the following system of ODE's, where \(x_1\) represents time, \(x_2\) represents the charge
	(at the capacitor) and \(x_3\) represents the electric current in the circuit.
	\begin{align}\label{CircuitSystem}
	  \begin{pmatrix}\dot x_1\\ \dot x_2\\ \dot x_3\end{pmatrix} &= F(\textbf{x}) = 
	  \begin{pmatrix}
	    1\\
	    x_3\\
	    - \bigl(x_2 - C V(x_1) + g(Cx_3)\bigr) \frac{1}{LC}
	  \end{pmatrix}
	\end{align}
	Here \(V(x_1)=sin(\omega x_1)\) is the forcing current and \(g(z)\) models the diode (for small currents) in the piecewise linear form
	\begin{align*}g(z) &=\frac{z+\betrag{z}}{2\alpha}+\frac{z-\betrag{z}}{2\beta}= 
	\begin{cases}
		\frac z\alpha &\text{if }z\geq 0\\
		\frac z\beta &\text{if }z<0 .
	\end{cases}
	\end{align*}
	We choose a set of constants that resembles those occurring in actual electric circuits: 
	\[L=10^{-6},\ C=10^{-13},\ \omega=3\cdot10^9,\ \alpha=2,\ \beta=0.00001\,.\]
	Moreover, we consider the initial conditions $x_1(0)=x_2(0)=x_3(0)=0$.
\begin{figure}[htp]
\centering
\subfigure[Current]{\includegraphics[height=180pt]{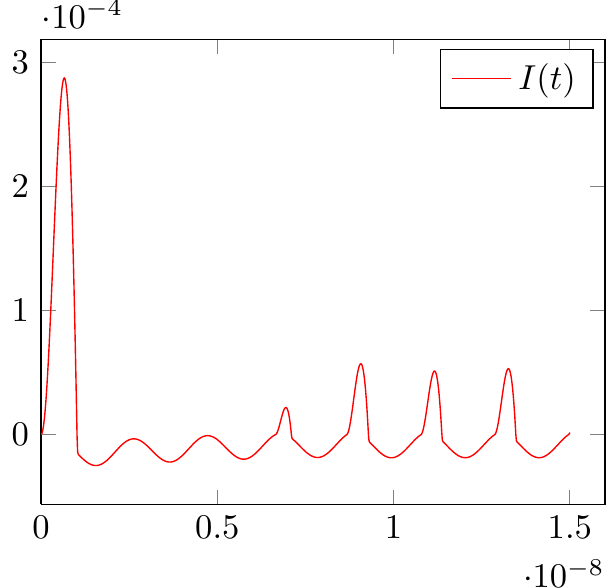}}
\hspace{5pt}
\subfigure[Charge]{\includegraphics[height=180pt]{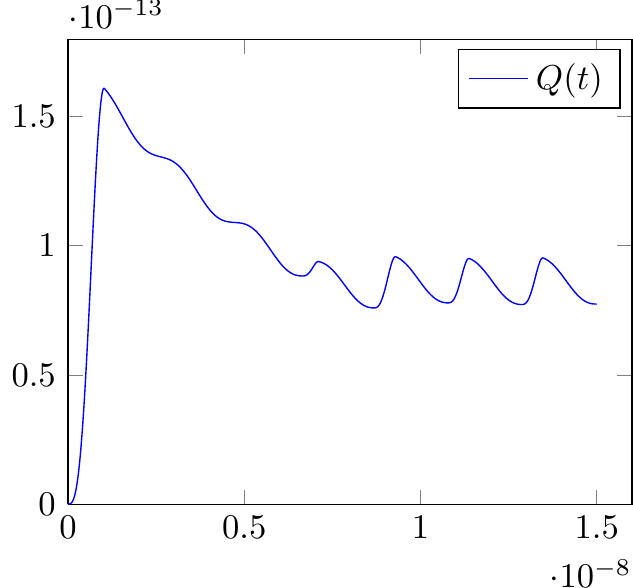}}
\caption{Solution of the ODE System}
\label{FIG_CircuitSolution}
\end{figure}
The result of the  numerical integration solution of \eqref{CircuitSystem} is depicted in Figure \ref{FIG_CircuitSolution}, (a) and (b).
	As one can see, the capacitor is initially charged over one cycle and discharged over a few more,
	before the solution adopts a periodic behavior. The solution trajectory changes its
	behavior every time the current changes its sign. As in the previous example, both methods display second order convergence, while in case of Romberg extrapolation the generalized method gains an order and the classical does not. 
\begin{figure}
\centering
\includegraphics[width=0.7\textwidth]{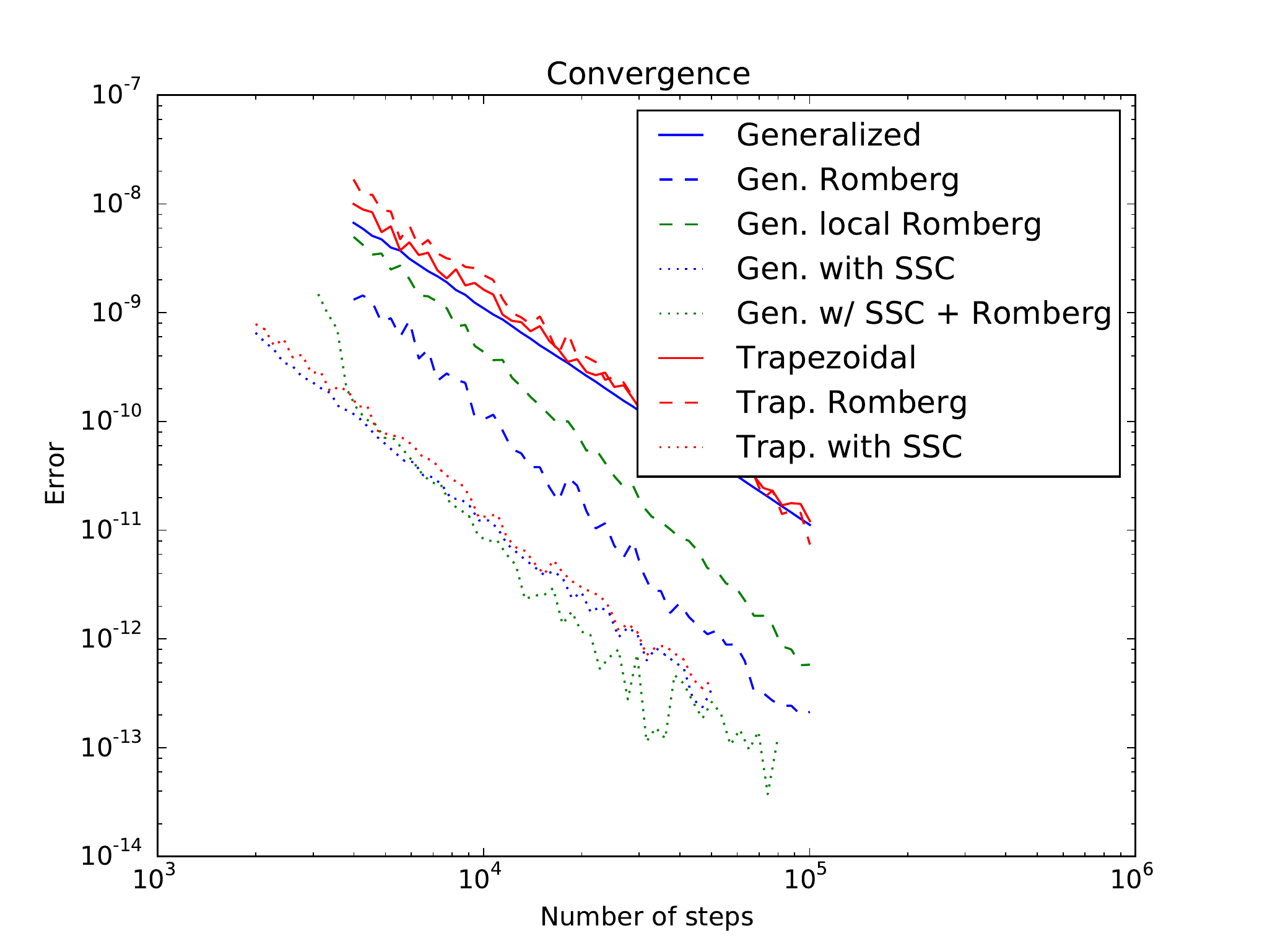}
\caption{Convergence Diode}
\label{conv_diode}
\end{figure}
In Figure \ref{conv_diode} we observe that adaptive time-stepping results in similar improvements to the convergence constant for the generalized method as for the classical method. Moreover it can be seen that adaptive time-stepping and extrapolation can be combined, which allows to get both a small constant and a third order convergence rate. 

\section{Final Remarks}
In the numerical results one can observe that the energy preservation is lost after extrapolation. This is due to the loss of time reversibility. It has to be investigated further, if this property can be restored. Furthermore, the current Lipschitz constants from \cite[Prop. 4.2.]{NewtonPL} are in some cases huge overestimations. It is desirable to sharpen these bounds. Also, one should investigate the possibility of automated scaling of the error norm using information from the structure of the piecewise linearization to reflect the dimensions of the components of the error. A long term goal is the extension of the method to piecewise smooth functions that are not continuous. A short term goal is the implementation of the method analyzed in this article. Each inner iteration of the method requires the solution of a piecewise linear system, e.g. by the solvers proposed in \cite{SGE,streubel2014abs,son}. We intend to deliver an efficient implementation for an integrated framework of piecewise linearization 
and ODE as well as equation solving in subsequent publications. 

\section*{Acknowledgements}

The work for the article has been partially conducted within the Research Campus
MODAL funded by the German Federal Ministry of Education and Research
(BMBF) (fund number \(05\text{M}14\text{ZAM}\)).

\bibliography{autodiff}
\bibliographystyle{alpha}

\end{document}